\documentclass[11pt]{amsart}
\usepackage{amssymb}
\usepackage{amsthm}
\usepackage[all,knot,arc]{xy}
\usepackage{tikz-cd}
\usepackage[abs]{overpic}
\usepackage{subfigure}
\usepackage{color}
\usepackage[breaklinks]{hyperref}
\hypersetup{backref,colorlinks=true,citecolor=blue,linkcolor=blue}
\newtheorem{theorem}{Theorem}[section]
\newtheorem{lemma}[theorem]{Lemma}
\newtheorem{proposition}[theorem]{Proposition}
\newtheorem{corollary}[theorem]{Corollary}
\newtheorem{question}[theorem]{Question}

 \usepackage{palatino}

\theoremstyle{definition}

\newtheorem{example}[theorem]{Example}

\newtheorem{remark}[theorem]{Remark}

\def\dfn#1{{\em #1}}

\newcommand{\R}{\ensuremath{\mathbb{R}}}
\newcommand{\Z}{\ensuremath{\mathbb{Z}}}
\newcommand{\C}{\ensuremath{\mathbb{C}}}

\def\co{\colon\thinspace}

\numberwithin{equation}{section}

\title[Embedding contact $3$--manifolds]{Embedding all contact $3$--manifolds in a fixed contact 5--manifold}

\author{John B. Etnyre}
\address{School of Mathematics \\ Georgia Institute of Technology}
\email{etnyre@math.gatech.edu}
\urladdr{\href{http://www.math.gatech.edu/~etnyre}{http://www.math.gatech.edu/\~{}etnyre}}

\author{Yanki Lekili}
\address{King's College London}
\email{yanki.lekili@kcl.ac.uk}
\urladdr{\href{https://nms.kcl.ac.uk/yanki.lekili/}{https://nms.kcl.ac.uk/yanki.lekili/}}

\begin{document}

\maketitle

\begin{abstract}
In this note we observe that one can contact embed all contact 3--manifolds into a Stein fillable contact structure on the twisted $S^3$--bundle over $S^2$ and also into a unique overtwisted contact structure on $S^3\times S^2$. These results are proven using ``spun embeddings'' and Lefschetz fibrations. 
\end{abstract}

 \section{Introduction}
A basic question in geometric topology is the embedding problem of manifolds: given two
(smooth) manifolds $M$ and $N$, can one find a smooth embedding of $M$ into $N$? In particular,
given $M$ what is a simple space $N$ in which $M$ can be embedded? Seeing an abstract manifold as a
submanifold of a simple space can provide concrete ways to describe the manifold as well as new
avenues to study the manifold. Whitney proved, by using transversality arguments and the famous ``Whitney trick'' that 
any $n$--manifold embeds in $\R^{2n}$ \cite{Whitney44}. This is the smallest possible
Euclidean space for which one can prove such a general theorem, though for specific manifolds and
specific values of $n$ one can do better.  For example, Hirsch \cite{Hirsch61} proved that all
oriented $3$--manifolds embed in $\R^5$, and Wall \cite{Wall65} then removed the orientability
assumption. There are homological (and other) obstructions to embedding $3$--manifolds in $\R^4$,
though Freedman \cite{Freedman82} did show that all homology $3$--spheres topologically, locally
flatly embed in $\R^4$. (To be clear, in this paper we will always be considering smooth embeddings
unless explicitly stated otherwise.) Since one cannot embed all $3$--manifolds in $\R^4$ one might
ask is there a, let us say, compact 4--manifold into which all $3$--manifolds embed? Shiomi
\cite{Shiomi91} answered this in the negative. So $\R^5$ can certainly be said to be the simplest
manifold into which all $3$--manifolds embed. Let us also mention that once an embedding is found, one could also ask the question of whether it is unique (up to isotopy). For example, the Schoenflies problem which concerns smooth
embeddings of $S^3$ into $\R^4$ is among the most famous open problems of topology. 

This paper is concerned with the contact analog of the above discussion: {\em What is the simplest
contact manifold into which all co-oriented contact $3$--manifolds embed?} (Throughout this paper all contact structures will be assumed to be co-oriented.) Recall that a smooth embedding $f\colon  M \to N$ is said to be a contact embedding of the contact manifold $(M,
\xi)$ into $(N, \xi' )$ if $df(TM) \pitchfork\xi'$ and $df(\xi) \subset \xi'$. In the case $M$ is
$1$-dimensional and $N$ is $3$-dimensional, a contact embedding $f: M^{1} \to N^{3}$ is also known as a transverse knot.

The analog of Whitney's theorem,
proven by Gromov \cite{Gromov86}, is that any contact $(2n+1)$--manifold can be embedded in the
standard contact structure on $\R^{4n+3}$. In particular, a contact $3$--manifold can
be embedded in $(\R^7,\xi_{std})$. As having an embedding in $\R^n$ is the same as having one in $S^n$ and we prefer to work with compact manifolds we will switch to consider spheres instead of Euclidean spaces. This result was reproved in \cite{Torres11}, based on work in \cite{Mori04}, using open book decompositions (which will be a key ingredient in the work presented in this paper too, but they will be used in a fairly different way). Showing that Hirsch's result does not generalize to the contact category we have the following result of Kasuya.
\begin{theorem}[Kasuya, 2014 \cite{Kasuya14}]\label{cantembed}
If $(M^{2n-1},\xi)$ embeds in $(N^{2n+1},\xi')$, $[M]=0$ in $H_{2n-1}(N;\Z)$ and $c_1(\xi')=0$, then $c_1(\xi)=0$.
\end{theorem}
Recall that a contact structure $\xi=\ker \alpha$ has a symplectic structure given by $d\alpha$ and
associated to this symplectic structure there is a unique homotopy class of a compatible complex structure on $\xi$. When referring to the Chern classes of $\xi$ we are using this complex structure. We will reprove and extend Theorem~\ref{cantembed} in Section~\ref{sec:obstruction}. 

From this theorem we see that there are many contact $3$--manifolds that do not embed in
$(S^5,\xi_{std})$. One might hope that a contact $3$--manifold embeds in $(S^5,\xi_{std})$ if and
only if it has trivial Chern class. While this is still an open question the following partial
results are obtained in a recent work of the first author and Furukawa.
\begin{theorem}[Etnyre-Furukawa, 2017 \cite{EtnyreFurukawa17}]
If $M$ is a $3$--manifold with no $2$--torsion in its second homology then an overtwisted contact structure embeds in $(S^5,\xi_{std})$ if and only if its first Chern class is zero.

If $M$ is $S^3$, $T^3$, or a lens space $L(p,q)$ with $p$ odd and $q$ arbitrary, or when $p$ is even and $q=1$ or $p-1$
respectively, then a contact structure on $M$ embeds in $(S^5,\xi_{std})$ if and only if it has trivial first Chern class. 
\end{theorem}
Kasuya \cite{Kasuya14} also proved that given a contact $3$--manifold with trivial first Chern
class, there is {\em some} contact structure on $\R^5$ into which it embeds, but the contact
structure on $\R^5$ could depend on the contact $3$--manifold and it may not necessarily be standard at infinity. Slightly extending Kasuya's proof by using recent work of Borman, Eliashberg, and Murphy \cite{BormanEliashbergMurphy} one can show the following.
\begin{theorem}[Etnyre-Furukawa, 2017 \cite{EtnyreFurukawa17}]
A contact $3$--manifold embeds into the unique overtwisted contact structure on $S^5$ if and only if it has trivial first Chern class. 
\end{theorem}

The overtwisted contact structure on $S^5$ is somewhat mysterious (though it has a simple description in terms of open book decompositions), so it would still be of great interest to have a similar theorem for embedding into $(S^5,\xi_{std})$.  It is at least known that any contact 3--manifold embeds in a symplectically fillable contact 5--manifold \cite{CasalsPresasSandon16} and also a hyper-tight contact 5--manifold \cite{Gironella17}, but it should be noted that in these results the contact 5--manifold depends on the contact 3--manifold being embedded.

Continuing with the search for a simple contact $5$--manifold into which all contact $3$--manifolds
can embed we must consider $5$-manifolds other than the $5$--sphere. The next simplest class of
manifolds to consider are products of spheres. Below we will see that a theorem analogous to
Theorem~\ref{cantembed} implies that there is no contact structure on $S^1\times S^4$ into which one
can embed all contact $3$--manifolds. In addition, if there is a contact structure on $S^2\times
S^3$ into which all contact $3$--manifolds embed, it must have first Chern class $\pm 2\in \Z\cong
H^2(S^2\times S^3)$, see Corollary~\ref{onlys2s3}. Indeed, we show that there is such a contact
structure:
\begin{theorem}\label{otembed}
There is (up to contactomorphism) a unique overtwisted contact structure $\xi_{ot}$ on $S^2\times S^3$ into which all contact $3$--manifolds embed with trivial normal bundle. 
\end{theorem}
\begin{remark}
At the moment the authors do not know if there are embeddings of oriented 3--manifolds into $S^2\times S^3$ with non-trivial normal bundles. 
\end{remark}
One would still like an embedding theorem with a ``nicer" or more ``standard" contact structure. To this end we ask:
\begin{question}
Is there a symplectically fillable contact structure on $S^2\times S^3$ into which all contact $3$--manifolds embed?
\end{question}
While we expect the answer to this question is yes, we can prove an analogous result for the unique non-trivial $S^3$-bundle over $S^2$. (Note that since $\pi_1(\mathrm{SO}(4))=\mathbb{Z}_2$, there are exactly two $S^3$-bundles over $S^2$). 
\begin{theorem}\label{mainembedthm}
	There is a Stein fillable contact structure $\xi$ on the (unique) twisted $S^3$-bundle over $S^2$ into which all contact $3$--manifolds embed. 
\end{theorem}
Our main approach to this theorem is via open book decompositions and Lefschetz fibrations. In particular, we explore contact ``spun embeddings" (see Section~\ref{mainembed}) where one embeds one manifold into another by embedding the pages of an open book for one into the pages for the other. Such embeddings have been studied under the name of spun knots \cite{Friedman05}, and they have even been studied in the context of contact geometry. Specifically by Mori (in \cite{Mori04} and unpublished work) and a few observations were made about spun embeddings in \cite{EtnyreFurukawa17} and how they relate to braided embeddings. For more recent results also see \cite{PancholiPanditSaha}.
\smallskip

\noindent
{\bf Acknowledgments:} We are grateful to Patrick Massot and the referee for many helpful suggestions. The first author was partially supported by NSF grant DMS-1608684. The second author was partially supported by the Royal Society (URF) and the NSF
grant DMS-1509141.

 \section{Lefschetz fibrations and open book decompositions}
In this section we recall the notion of an open book decomposition and its relation to contact structures on manifolds. We then recall some basic facts about Lefschetz fibrations and symplectic manifolds. We end this section by reviewing overtwisted contact structures in high dimensions. 

\subsection{Open book decompositions}
Given a manifold with boundary $X$ and a diffeomorphism $\phi:X\to X$ whose support is contained in the interior of the manifold one can consider the mapping torus 
\[
T_\phi= X\times [0,1]/\sim
\]
where $\sim$ is the relation $(x,0)\sim (\phi(x),1)$. It is easy to see that $\partial
T_\phi=(\partial X)\times S^1$ and thus there is an obvious way to glue $(\partial X)\times D^2$ to
$T_\phi$ to obtain a manifold, which we denote by $M_{(X,\phi)}$. We say that a manifold $M$ has an \dfn{open book decomposition} $(X,\phi)$ if $M$ is diffeomorphic to $M_{(X,\phi)}$. We call $\phi$ the \dfn{monodromy} of the open book. 

Notice that the manifold $B=(\partial X)\times \{0\}$, where $0$ is the center of $D^2$, is a sub-manifold of $M_{(X,\phi)}$ with neighborhood $N=B \times D^2$ and that $M_{(X,\phi)}-B$ fibers over the circle
\[
\pi: (M_{(X,\phi)}-B)\to S^1
\]
 so that the fibration on $N-B$ is simply projection to the $\theta$ coordinate of $D^2$ (where $D^2$ is given polar coordinates $(r,\theta)$). So the image of $B$ in $M$ also has these properties, we denote this submanifold of $M$ by $B$ too, and the fibration of its complement by $\pi$. We call $(B,\pi)$ an \dfn{open book decomposition} of $M$ too. Notice that, up to diffeomorphism, $(B,\pi)$ and $(X,\phi)$ each determine the other, so we can describe open books using either definition depending on the situation. We call $B$ the \dfn{binding} of the open book and $\overline{\pi^{-1}(\theta)}$ a \dfn{page} of the open book for any $\theta\in S^1$. 
 \begin{example}\label{idmonodromy}
 As a simple example consider any manifold $X$ and $\phi=id_X$. One may easily see that $M_{(X,\phi)}$ is diffeomorphic to $X\times [0,1/2]$ glued to $X\times [1/2,1]$ in the obvious way. (That is, $X\times \{0\}$ is glued to $X\times \{1\}$ by the identity map, the two copies of $X\times \{1/2\}$ are glued by the identity map, and $(\partial X)\times [0,1/2]$ is glued to $(\partial X)\times [1/2,1]$ by the map $(x,t)\mapsto (x,1-t)$.) But of course this is simply the double $D(X\times [0,1])$ of $X\times [0,1]$. 
 
As a concrete example, consider $X$ the $D^2$-bundle over $S^2$ with Euler number $n$. Then $X\times [0,1]$ is the $D^3$-bundle over $S^2$ with Euler number $n\mod 2$ (there are only two such bundles determined by the residue mod $2$). The double of $S^2\times D^3$ is clearly $S^2\times S^3$ and the double of the twisted $D^3$-bundle over $S^2$ is clearly the twisted $S^3$-bundle over $S^2$, $S^2\widetilde\times S^3$.
 \end{example}

The fundamental connection between open books and contact structures is given in the following theorem.
\begin{theorem}[Thurston-Winkelnkemper 1975, \cite{ThurstonWinkelnkemper75} for $n=1$ and Giroux 2002, \cite{Giroux02} for arbitrary $n$]\label{obsupport}
Given a compact $2n$--manifold $X$, a diffeomorphism $\phi:X\to X$ with support contained in the interior of $X$ and a $1$-form $\beta$ on $X$ such that
\begin{enumerate}
\item $d\beta$ is a symplectic form on $X$,
\item the Liouville vector field $v$ defined by 
\[
\iota_v d\beta=\beta
\]
points out of $\partial X$, and
\item $\phi^*d\beta=d\beta$,
\end{enumerate}
then $M_{(X,\phi)}$ admits a unique, up to isotopy, contact structure $\xi_{(X,\phi)}$ whose defining $1$-form $\alpha$ satisfies 
\begin{enumerate}
\item $\alpha$ is a positive contact from on the binding of the open book and
\item $d\alpha$ is a symplectic form when restricted to each page of the open book. 
\end{enumerate}
\end{theorem}
The contact structure guaranteed by the theorem is said to be \dfn{supported by} or \dfn{compatible with} the open book. This terminology is due to Giroux \cite{Giroux02} and the uniqueness part of the theorem, even in dimension $3$, was established by Giroux. 

As we will need the details of the proof for our work we sketch Giroux's proof of this theorem here.
\begin{proof}[Sketch of proof]
We begin by assuming that $\phi^*\beta=\beta -dh$ for some function $h:X\to \R$. Of course, by adding a constant we can assume that $h(x)$ is bounded above by a large negative constant. We notice that $\widetilde{\alpha}=dt+\beta$ is a contact form on $X\times \R$ where $\R$ has coordinate $t$. There is an action of $\Z$ on $X\times \R$ generated by the diffeomorphism $(x,t)\mapsto (\phi(x),t+h(x))$. Clearly $\widetilde{\alpha}$ is invariant under this action and descends to a contact form $\alpha$ on $(X\times \R)/\Z$ which is canonically diffeomorphic to $T_\phi$. One may easily find a non-decreasing function $f(r)$ that is equal to $r^2$ near $0$, and a positive but non-increasing function $g(r)$ such that the contact form 
	\[ g(r) \beta|_{\partial X} + f(r)\, d\theta \] on $(\partial X)\times D^2$, where $D^2$ is given polar coordinates, can be glued to the contact from $\alpha$ on $T_\phi$ to get a contact form on $M_{(X,\phi)}$, {\em cf.\ }\cite{Geiges08}.

Now if $\phi^*\beta-\beta$ is not exact then let $\eta$ be the vector field on $X$ satisfying
\[
\iota_\eta d\beta=\beta-\phi^*\beta,
\]
and $\psi_t:X\to X$ its flow. We note for future reference that as $\phi$ is equal to the identity near $\partial X$ we know $\eta$ is zero there and $\phi_t$ is also supported on the interior of $X$. A computation shows  
that $\phi'=\phi\circ \psi_1$ satisfies $(\phi')^*\beta=\beta+dh$ and of course $M_{(X,\phi)}$ is diffeomorphic to $M_{(X,\phi')}$ since $\phi$ is isotopic to $\phi'$. 

One may easily see that the constructed contact structure is compatible with the open book. As the details are not needed here, we leave the uniqueness of the compatible contact structure to the reader.   For a somewhat different proof of this theorem see \cite{Giroux17}. 
\end{proof}
\begin{remark}\label{usefulcomments}
We make a few observations for use later. First notice that one may use the from $k\, dt+ \beta$, for any $k>0$, to construct the contact structure on the mapping torus part of $M_{(X,\phi)}$. In addition, notice that when the monodromy is the identity, then construction in the proof is particularly simple. 

Finally we observe that if $X$ is $2$-dimensional then we can find a 1-parameter family of 1-forms $\beta_t$ on $X$ interpolating between $\beta$ and $\phi^*\beta$. Then one easily checks that for $k$ large enough $k\, dt + \beta_t$ is a contact form on the mapping torus part of $M_{(X,\phi)}$ that can be extended over the binding as was done in the proof above.  
\end{remark}

\begin{example}\label{examplestein}
Continuing Example~\ref{idmonodromy} let $X$ be a Stein domain in the Stein manifold $X'$. Then the Stein structure on $X$ gives a canonical 1--form $\beta$ on $X$ as in Theorem~\ref{obsupport}. Let $\xi$ be the contact structure on $M_{(X,id_X)}$ supported by the open book $(X,id_X)$. From Example~\ref{idmonodromy} we know that $M_{(X,id_x)}$ is the boundary of $X\times D^2$. We note that $X\times D^2$ can be given the structure of a Stein domain (indeed on $X\times \C$ consider $\Psi(x,z)=\psi(x)+\|z\|^2$, where $\psi:X\to \R$ is a strictly pluri-subharmonic function. It is clear that $\Psi$ is also strictly pluri-subharmonic and will define a domain diffeomorphic to $X\times D^2$). It is also easy to check that the Stein domain is a filling of $(M_{(X,id_X)},\xi)$ (cf. \cite[Prop. 3.1]{DGvK}). 
\end{example}

Giroux also proved the following ``converse" to the Theorem~\ref{obsupport}.
\begin{theorem}[Giroux 2002, \cite{Giroux02}]
Every contact structure on a closed $(2n+1)$--manifold is supported by some open book decomposition. 
\end{theorem}

\subsection{Lefschetz fibrations}

A \dfn{Lefschetz fibration} of an oriented 4--manifold $X$ is a map $f\co X\to F$ where $F$ is an oriented surface and $df$ is surjective at all but finitely many points $p_1,\ldots, p_k$, called \dfn{singular points},  each of which has the following local model: each point $p_i$ has a neighborhood $U_i$ that is orientation preserving diffeomorphic to an open set $U_i'$ in $\C^2$, $f(p_i)$ has a neighborhood $V_i$ that is orientation preserving diffeomorphic to $V_i'$ in $\C$ and in these local coordinates $f$ is expressed as the map $(z_1,z_2)\mapsto z_1z_2$. We say $f\co:X\to F$ is an \dfn{achiral Lefschetz fibration} if there is a map $f\co X\to F$ as above except the local charts expressing $f$ as $(z_1,z_2)\mapsto z_1z_2$ do not have to be orientation preserving. 

We state a few well known facts about Lefschetz fibrations, see for example \cite{OzbagciStipsicz04}. Let $f\co X\to F$ be a Lefschetz fibration with $\{p_1,\ldots, p_k\}\subset X$ the set of singular points. 
\begin{enumerate}
\item Setting $F'=F\setminus f(\{p_1,\ldots, p_k\})$ and $X'=f^{-1}(F')$ the map $f|_{X'}\co X'\to F'$ is a fibration with fiber some surface $\Sigma$. 

\item Fix a point $x\in F'$ and for each $i=1,\ldots, k$, let $\gamma_i$ be a path in $F$ from $x$ to $f(p_i)$ whose interior is in $F'$. Then there is an embedded curve $v_i$ on $F_x=f^{-1}(x)$ that is homologically non-trivial in $F_x$ but is trivial in the homology of $f^{-1}(\gamma_i)$. The curve $v_i$ is called the \dfn{vanishing cycle of $p_i$} (though it also depends on the arc $\gamma_i$). We will assume that $\gamma_i\cap \gamma_j=\{x\}$ for all $i\not=j$. 

\item For each $i$ let $D_i$ be a disk in $F$ containing $f(p_i)$ in its interior, disjoint from the $\gamma_j$ for $j\not=i$, and intersecting $\gamma_i$ in a single arc that is transverse to $\partial D_i$. The boundary $\partial f^{-1}(D_i)$ is a $\Sigma$-bundle over $S^1$. Identifying a fiber of $ f^{-1}(\partial D_i)$ with $\Sigma$ using $\gamma_i$, the monodromy of $ f^{-1}(\partial D_i)$ is given by a right-handed Dehn twist about $v_i$. 

\item If the surface $F$ is the disk $D^2$ then let $D$ be a disk containing $x$ and intersecting each $\gamma_i$ in an arc transverse to $\partial D$. The manifold $X$ can be built from $D^2\times \Sigma=f^{-1}(D)$ by adding a $2$--handle for each $p_i$ along $v_i$ sitting in $F_{q_i}=f^{-1}(q_i)$ where $q_i=\partial D\cap \gamma_i$ with framing one less than the framing of $v_i$ given by $F_{q_i}$ in $\partial D^2\times \Sigma$. Conversely, any 4-manifold constructed from $D^2\times \Sigma$ by attaching $2$--handles in this way will correspond to a Lefschetz fibration.

\item If $f\co X\to F$ is an achiral Lefschetz fibration then the above statements are still true but for an ``achiral" singular point the monodromy Dehn twist is left-handed and the $2$--handle is added with framing one greater than the surface framing. 
\end{enumerate}

The following theorem is well-known and provides a way to study symplectic/contact topology in low-dimensions via Lefschetz fibrations, in turn, via the theory of mapping class groups of surfaces.

\begin{theorem}\label{leftosym}
Suppose that $X$ is a 4--manifold that admits a Lefschetz fibration $f:X\to D^2$. If the fiber of $f$ is not null-homologous then $X$ admits a symplectic structure $\omega$ in which each fiber of $f$ is symplectic. If moreover, the fibers of $f$ are surfaces with boundary, and the vanishing cycles are non-separating, $\omega$ can be taken to be exact and $(X,\omega=d\beta)$ a strong filling of a contact structure $\xi$ on $\partial X$. 

There is an open book $(B,\pi)$ for $\partial X$ supporting $\xi$ that can be described as follows:
\begin{enumerate}
\item We may assume $D^2$ is the unit disk in $\C$ and $0\in D^2$ is a regular value of $f$. The binding $B$ is $\partial f^{-1}(0)$. 
\item The projection $\pi:(\partial X-B)\to S^1$ is simply the composition of $f$ restricted to $\partial X-B$ and projection to the $\theta$ coordinate of $D^2$ (where $D^2$ is given polar coordinates $(r,\theta)$). 
\item There is a subdisk $D'$ of the base $D^2$ containing $0$ and all the critical points of $f$ such that $f^{-1}(D') \cap \partial X$ is a neighborhood of the binding and if $v$ is the Liouville field for $\omega=d\beta$ then the contact form $\alpha=\iota_v\omega$ is given by $d\phi +r^2\, d\theta$ on each component of $f^{-1}(D')$.
\end{enumerate}
\end{theorem}

\begin{proof}[Sketch of proof] In the case when the fibers are closed, and there exists a class in $H^2(X)$ that integrates to a positive number along the fiber, the existence of a symplectic form is due to Gompf \cite{gompf}, which in turn relies on a classical argument of Thurston \cite{thurston}. 
	The case when the fibers are surfaces with boundary (or more generally are exact symplectic manifolds) and the vanishing cycles are non-separating (or more generally exact Lagrangian spheres) is studied extensively by Seidel in \cite[Chapter~3]{seidel}. Note that on a punctured surface equipped with an exact symplectic form, we can isotope any non-separating curve to an exact Lagrangian (unique up to Hamiltonian isotopy), then by \cite[Lemma~16.8]{seidel}, we can construct an exact symplectic structure on $X$ and an exact Lefschetz fibration $f$ on $X$. The properties (1)-(3) listed above are consequences of the definition of an exact Lefschetz fibration given in \cite[Section~15(a)]{seidel}. In particular, for the property (3) see \cite[Remark~15.2]{seidel}, where the triviality of the symplectic connection along the horizontal boundary can be arranged in our case since the base of our exact Lefschetz fibration, being $D^2$, is contractible.
\end{proof}

\begin{proposition}\label{exampleprop}
Let $\Sigma$ be a surface of genus $g\geq 3$ and $\gamma_1,\ldots, \gamma_{2g+1}$ the curves shown
in Figure~\ref{fig:surface}. Then the disk bundle over $S^2$ with Euler number $-3$ has a Lefschetz
fibration with vanishing cycles $\{\gamma_1,\ldots, \gamma_{2g}, \gamma_{2g+1}\}$, and the disk bundle over $S^2$ with Euler number $-4$ has a Lefschetz fibration with vanishing cycles $\{\gamma_1,\ldots, \gamma_{2g}, \gamma_{2g+2}\}$. The Lefschetz fibration gives
these manifolds an exact symplectic structure that is independent of the genus of $\Sigma$. Similarly $D^4$ has a Lefschetz
fibration with vanishing cycles $\{\gamma_1,\ldots, \gamma_{2g}\}$. 
\end{proposition}
\begin{figure}[htb]{
\begin{overpic}
{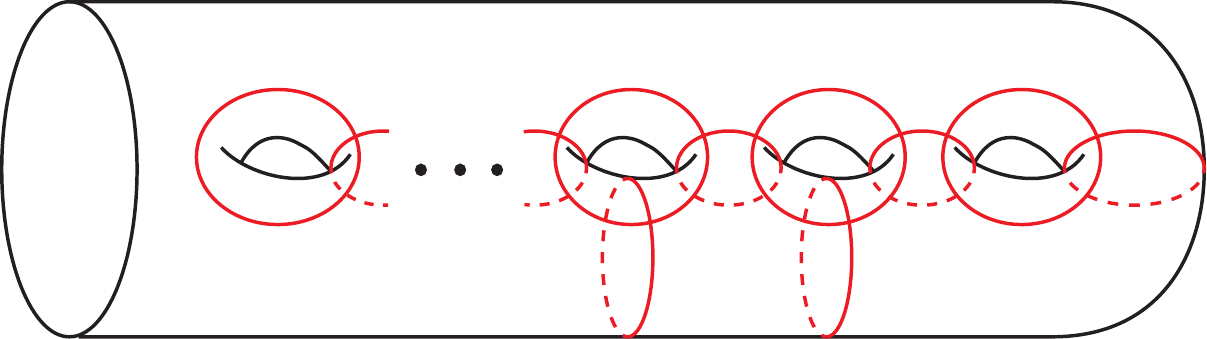}
\put(326,64){$\gamma_1$}
\put(291, 76){$\gamma_2$}
\put(262,64){$\gamma_3$}
\put(236,76){$\gamma_4$}
\put(206, 64){$\gamma_5$}
\put(179, 76){$\gamma_6$}
\put(76,76){$\gamma_{2g}$}
\put(249,18){$\gamma_{2g+1}$}
\put(191,18){$\gamma_{2g+2}$}
\end{overpic}}
\caption{A surface $\Sigma$ of genus $g$ with one boundary component and $2g+2$ curves marked.}
\label{fig:surface}
\end{figure}
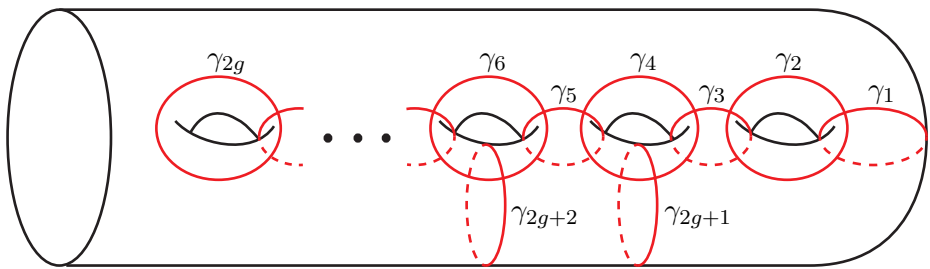

\begin{proof}
One may easily describe a handle presentation for the 4--manifold described by the Lefschetz
fibration in the proposition, see \cite{OzbagciStipsicz04}. The handlebody picture in
Figure~\ref{fig:stein}  is one such description when the genus is $3$.  There is an obvious
extension of this picture to the higher genus case.  The 1 and $2$--handles for higher genus
surfaces all cancel and do not interact with $\gamma_{2g+1}$ and $\gamma_{2g+2}$.

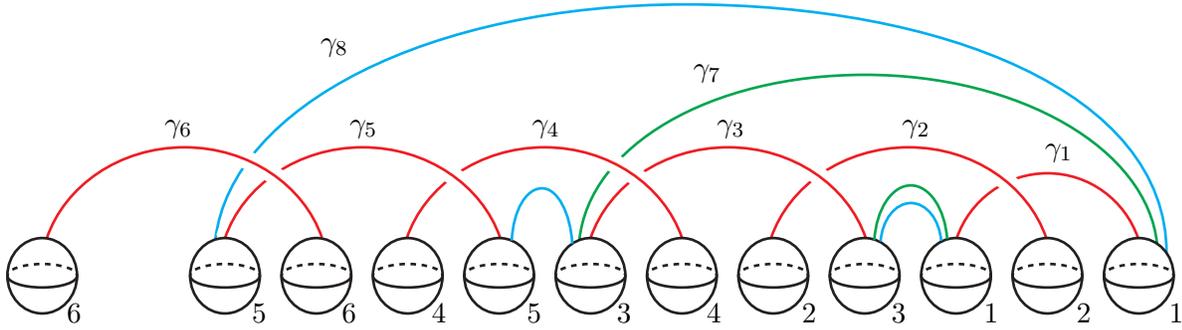
\begin{figure}[htb]{
\begin{overpic}
{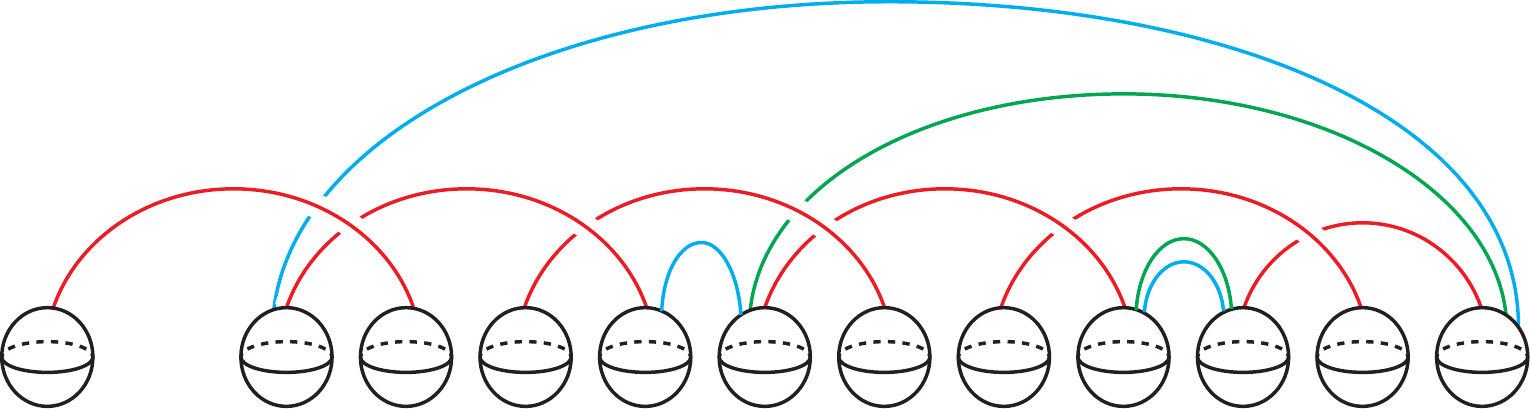}
\put(440,-3){$1$}
\put(405, -3){$2$}
\put(370,-3){$1$}
\put(335,-3){$3$}
\put(301,-3){$2$}
\put(265,-3){$4$}
\put(231,-3){$3$}
\put(197,-3){$5$}
\put(161,-3){$4$}
\put(127,-3){$6$}
\put(93,-3){$5$}
\put(23,-3){$6$}
\put(260,90){$\gamma_{7}$}
\put(119,100){$\gamma_{8}$}
\put(60,69){$\gamma_6$}
\put(130,69){$\gamma_5$}
\put(199,69){$\gamma_4$}
\put(269,69){$\gamma_3$}
\put(339,69){$\gamma_2$}
\put(393,60){$\gamma_1$}
\end{overpic}}
\caption{Handle presentation describing manifolds in Proposition~\ref{exampleprop} when the genus is $3$. All $2$--handles have framing $-1$. The 1--handles are paired as indicated by the numbers.}
\label{fig:stein}
\end{figure}
Notice that all the 1 and $2$--handles cancel if we do not have the vanishing cycles $\gamma_{2g+1}$
and $\gamma_{2g+2}$. Thus the manifold is $B^4$. After this cancellation the curves $\gamma_{2g+1}$
and $\gamma_{2g+2}$ are both unknots with the former having framing $-3$ and the later having framing $-4$. Thus the first Lefschetz fibration in the theorem gives disk bundle over the sphere with Euler number $-3$ and the Lefschetz fibration with $\gamma_{2g+2}$ replacing $\gamma_{2g+1}$ results in the disk bundle over the sphere with Euler number $-4$.

It is well known how to turn these handlebody diagrams into Stein handlebody diagrams, see again \cite{OzbagciStipsicz04}, and thus our manifolds have an exact symplectic structure. As the $1$ and $2$--handles for higher genus surfaces symplectically cancel we see this structure is independent of the genus. 
\end{proof}

\subsection{Overtwisted contact structures}\label{sec:ot}
Recall an \dfn{almost contact structure} on a $(2n+1)$--dimensional manifold $M$ is a reduction of
its structure group to $U(n)\times \mathbf{1}$. This is easily seen to be a necessary condition for
a manifold to admit a contact structure. We will be mainly considering contact structures in
dimension $5$ so a reduction of the structure group corresponds to a section of the $SO(5)/U(2)$-bundle associated to the tangent bundle of $M$. It is known that $SO(5)/U(2)$ is diffeomorphic to
$\C P^3$, for this and other facts below see for example \cite{Geiges08}, and thus the only (and
primary) obstruction to the existence of an almost contact structure lives in $H^3(M;\Z)$ and turns
out to be the integral second Stiefel-Whitney class (which vanishes if and only if the ordinary
second Stiefel-Whitney has an integral lift). Moreover, two almost contact structures are homotopic
if and only if they are homotopic on the $2$--skeleton of $M$. If $H^2(M;\Z)$ has no $2$-torsion then an almost contact structure is determined up to homotopy by its first Chern class $c_1(\xi)$. 

Borman, Eliashberg, and Murphy \cite{BormanEliashbergMurphy} defined a notion of overtwisted contact
structure in all dimensions and proved a strong version of the $h$-principle for such structures, in
particular they showed that any almost contact structure can be homotoped to a unique overtwisted
contact structure. The precise definition of an overtwisted contact structure will not be needed here, but see
\cite{BormanEliashbergMurphy} for the definition and \cite{CasalsMurphyPresas} for alternate
(possibly simpler) definitions.  Here, we content ourselves with stating the main theorem that we need. 

\begin{theorem}[Borman, Eliashberg, and Murphy 2014, \cite{BormanEliashbergMurphy}]\label{otex}
Let $\eta$ be an almost contact structure on a manifold $M$ that defines a contact structure on some
neighborhood of a closed (possible empty) subset $A$. Then there is a homotopy of $\eta$, fixed on
$A$, through almost contact structures to an overtwisted contact structure $\xi$ on $M$. Moreover, any other contact structure $\xi'$ that agrees with $\xi$ on $A$ and is overtwisted in the complement of $A$ is isotopic, relative to $A$, to $\xi$. 
\end{theorem}

 \section{Topological embeddings}\label{mainembed}
 
 In this section we discuss a general procedure for constructing embeddings using open book decompositions. 
 
\subsection{Smooth embeddings via open books}
The simplest way to embed one manifold into another using open book decompositions is the following. 
 \begin{lemma}\label{simplespin}
 Given two open book decompositions $(X,\psi)$ and $(Y,\phi)$ and a family of embeddings $f_t:Y\to X$, $t\in[0,1]$, such that
 \begin{enumerate}
 \item each $f_t$ is proper,
 \item $f_t$ is independent of $t$ near $\partial Y$, and
 \item $\psi\circ f_1=f_0\circ \phi$,
 \end{enumerate}
 then there is a smooth embedding of $M_{(Y,\phi)}$ into $M_{(X,\psi)}$.
 \end{lemma}
 \begin{proof}
 The last condition on $f_t$ guarantees that the embedding
 \[
 Y\times [0,1]\to X\times [0,1]: (y,t)\mapsto (f_t(y), t)
 \]
 descends to an embedding of mapping torus $T_\phi$ into $T_\psi$. The first two conditions on $f_t$ guarantee that the embedding 
 \[
 (\partial Y)\times D^2\to (\partial X)\times D^2: (y,z)\mapsto (f_0(y), z)
 \]
 extends the embedding of the mapping tori to an embedding of $M_{(Y,\phi)}$ into $M_{(X,\psi)}$.
 \end{proof}
 This operation has been extensively studied in the context of knot theory. In particular, given a properly embedded arc $c$ in $D^3$ one obtains an embedding of $S^2$ (that is the manifold with open book having page an arc and monodromy the identity) into $S^4$ obtained from the above lemma (where $f_t$ is independent of $t$) when thinking of $S^4$ as $M_{(D^3,id_{D^3})}$. If $c$ is obtained from a knot $K$ in $S^3$ by removing a small ball about a point on $K$ then knotted $S^2$ is called the \dfn{spun knot} in $S^4$ or is said to be obtained by \dfn{spinning} $K$. If one has a nontrivial family of embeddings of an arc into $D^3$ then the result is frequently called a twist spun knot, \cite{Friedman05}. Given this history we call the embedding constructed in the lemma above a \dfn{spun embedding}. 
 
We also notice the converse to the above lemma holds. If $W$ has an open book $(B,\pi)$ and an embedding $M\to W$ is transverse to $B$ and the pages of the open book, then there is an induced open book $(B',\pi')$ on $M$ such that the embedding can be described in terms of a spun embedding. 

One can embed open books in a more interesting way by noting that non-trivial diffeomorphisms of a submanifold can be induced by isotopy. 

\begin{lemma}\label{maintopembed}
Suppose that $f:X^4\to D^2$ is a (possibly achiral) Lefschetz fibration with fiber surface $\Sigma$ and vanishing cycles $v_1,\ldots, v_n$.  If $M$ is a $3$--manifold described by an open book $(\Sigma, \phi)$ where $\phi$ can be written as a composition of right and left handed Dehn twists about the $v_i$, then there is a spun embedding of $M$ into $M_{(X,\psi)}$, where $\psi$ is any diffeomorphism of $X$ (equal to the identity in a neighborhood of $\partial X$). 
\end{lemma}
\begin{proof}
Choose a regular value $x$ of $f$ in $D^2$ so that $f^{-1}(x)$ is in the region where $\psi$ is the identity. Also choose paths $\gamma_i$ from $x$ to the critical points $p_i$ of $f$ which are disjoint apart from at $x$ that realize the vanishing cycles. For each $i$ we can choose a disk $D_i$ that is a small neighborhood of $\gamma_i$ (and disjoint from the other $p_j$). Let $c_i=\partial D_i$ isotoped slightly so that it goes through $x$ and so that all the $c_i$ are tangent at $x$. The (regular) fibers of $f$ are oriented surfaces so an orientation on $c_i$ makes $f^{-1}(c_i)$ an oriented $3$-manifold and of course if $c_i$ is oriented as the boundary of $D_i$ (which in turn is oriented as a subset of $D^2$) then $f^{-1}(c_i)$ is the $\Sigma$-bundle over $S^1$ with monodromy a right-handed Dehn twist $\tau_{v_i}$ about $v_i$ if $p_i$ is an ordinary Lefschetz critical points (and the inverse of this if it is an achiral critical point). Considering the opposite orientation $-c_i$ on $c_i$ we see that $f^{-1}(-c_i)$ is the $\Sigma$-bundle over $S^1$ with monodromy a left-handed Dehn twist $\tau_{v_i}^{-1}$ about $v_i$ if $p$ is an ordinary Lefschetz point (and the inverse of this if it is an achiral critical point). Now let $c_i'$ be $-c_i$ isotoped near $x$ so that it is positively tangent to $c_i$ at $x$ (in particular $c_i'$ will be an immersed curve), see Figure~\ref{fig:curves}. Clearly there is an immersion of the previous surface bundle with image $f^{-1}(c_i')$.
\begin{figure}[htb]{
\begin{overpic}
{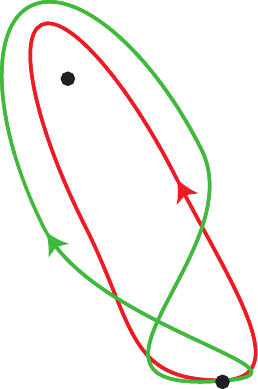}
\put(-2,50){$c_i'$}
\put(40,60){$c_i$}
\put(61, -7){$x$}
\put(18,80){$p_i$}
\end{overpic}}
\caption{The curves $c_i$ and $c_i'$ (slightly offset for clarity).}
\label{fig:curves}
\end{figure}
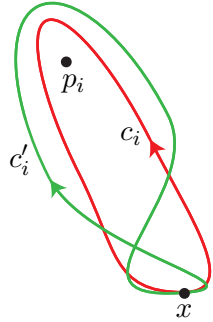

We will construct our embedding in three steps.

{\bf Step 1(preliminary embedding of mapping tori):} Note that given $\phi:\Sigma\to \Sigma$ which can be written as a composition of right and left handed Dehn twists about $v_i$, there is a path $\gamma:[0,1]\to D^2$ that is a composition of the paths $c_i$ and $c_i'$ so that the mapping torus $T_\phi$ immerses in $X$ with image $f^{-1}(\gamma)$. In other words, there is a map
\[
i:\Sigma\times [0,1]\to X
\]
that when composed with $f$ is $\gamma$ and induces an immersion on $T_\phi$. (More precisely, the pull back of $f:X\to D^2$ by $\gamma$ is a $\Sigma$-bundle over $S^1$ with monodromy $\phi$, now cutting this bundle along a fiber gives the desired immersion.) Notice that the map
\[
\widetilde{e}:\Sigma\times [0,1]\to X\times [0,1]: (p,t)\mapsto (i(p),t)
\]
is an embedding and induces an embedding $e:T_\phi\to T_\psi$ (here, it is important that $f^{-1}(x)$ is in the region where $\psi$ is the identity). 

{\bf Step 2 (fix the embedding near the boundary):} We would now like to extend our embedding over the binding. To do this we need to make $\widetilde{e}$ restricted to each boundary component be independent of $t$. To this end  notice that $e(\partial \Sigma\times\{t\})$ maps into $\partial f^{-1}(\text{int} (D^2))$ which is a union of solid tori (one for each boundary component of $\Sigma$). We will focus on one of these solid tori which we call $S$, the arguments for the others being the same. If $[a,b]\times S^1$ is a neighborhood of a boundary component of $\Sigma$ corresponding to $S$ then since $f$ is a trivial fibration restricted to the neighborhood $N$ of $S$ we see that $N= [a,b]\times S^1\times \text{int} (D^2)$ with $f$ being projection to the last factor. Thus $\widetilde{e}$ restricted to this neighborhood is simply 
\[
([a,b]\times S^1)\times [0,1]\to ([a,b]\times S^1\times\text{int} (D^2))\times [0,1]: (s,\theta, t)\mapsto (s,\theta,\gamma(t),t).
\]

We now slightly enlarge $\Sigma$ and $X$ by adding collar neighborhoods and use these extensions to make $\widetilde{e}$ independent of $t$ near the boundary components. Specifically, we add $[b,c]\times \partial \Sigma$ to $\Sigma$ and $[b,c]\times \partial X$ to $X$. Our monodromy maps can be extended by the identity and then this does not change the diffeomorphism type of $T_\phi$ and $T_\psi$. For each $t$, we let $\gamma_t:[b,c]\to \text{int}(D^2)$ be a straight line homotopy from $\gamma(t)$ to $x$. We now extend $\widetilde{e}$ by the map
\[
([b,c]\times  S^1)\times [0,1]\to ([b,c]\times S^1\times\text{int} (D^2))\times [0,1]: (s,\theta, t)\mapsto (s,\theta, \gamma_t(s),t).
\]
This extended $\widetilde{e}$ now is independent of $t$ near $\partial \Sigma$ and thus descends to an embedding $e:T_\phi\to T_\psi$ that for each boundary component $S^1$ of $\Sigma$, sends $S^1\times S^1$  to $S^1\times \{x\}\times S^1$. 

{\bf Step 3 (extend embedding over binding):} To form $M_{(\Sigma,\phi)}$ from $T_\phi$ we glue in a $S^1\times D^2$ to each boundary component $S^1$ of $\Sigma$ and similarly for $M_{(X,\psi)}$. Notice that $\partial (S^1\times D^2)$ maps into the part of $\partial (\partial X\times D^2)$ written as $\{c\}\times S^1\times \text{int}(D^2)\times S^1$ in the above coordinates. Thus we can extend $e$ over each $S^1\times D^2$ by the map
\[
S^1\times D^2\to S^1\times \text{int}(D^2)\times D^2: (\theta, p)\mapsto (\theta, x, p)
\]
to get an embedding of 
$M_{(\Sigma,\phi)}$ to get an embedding $M=M_{(\Sigma,\phi)}\to M_{(X,\psi)}$.
\end{proof}

\subsection{Embedding oriented $3$--manifolds in $S^5$}
Recall from Proposition~\ref{exampleprop} that there is a Lefschetz fibration of $B^4$ with vanishing cycles generating the hyperelliptic mapping class group. Thus our next result immediately follows from Lemma~\ref{maintopembed}.
\begin{proposition}
If $M$ is a $3$--manifold supported by an open book with hyperelliptic monodromy (that is the monodromy is a composition of positive and negative Dehn twists about $\gamma_1,\ldots, \gamma_{2g}$), then $M$ embeds in $S^5$. In particular if $M$ is obtained as the $2$--fold cover of $S^3$ branched over a link, then it has a spun embedding into $S^5$
\end{proposition}

We can slightly strengthen the above result as follows. 
\begin{proposition}\label{thisone}
Suppose $M$ is a 3--manifold supported by an open book with monodromy a composition of positive and negative Dehn twists about $\gamma_1,\ldots, \gamma_{2g}$ and $\gamma_{2g+2}$. Then $M$ embeds in $S^5$.
\end{proposition}
\begin{proof}
Take the Lefschetz fibration over $D^2$ with fiber a surface of genus $g$ and vanishing cycles $\gamma_1,\ldots \gamma_{2g}$. All will be ordinary vanishing cycles except for $\gamma_3$ which will be achiral. To this we add two achiral vanishing cycles along $\gamma_{2g+2}$ and two more copies of $\gamma_6$, both achiral. See Figure~\ref{fig:stein2} for the genus $3$ case, for higher genus the extra 1 and 2--handles will all simply cancel. 
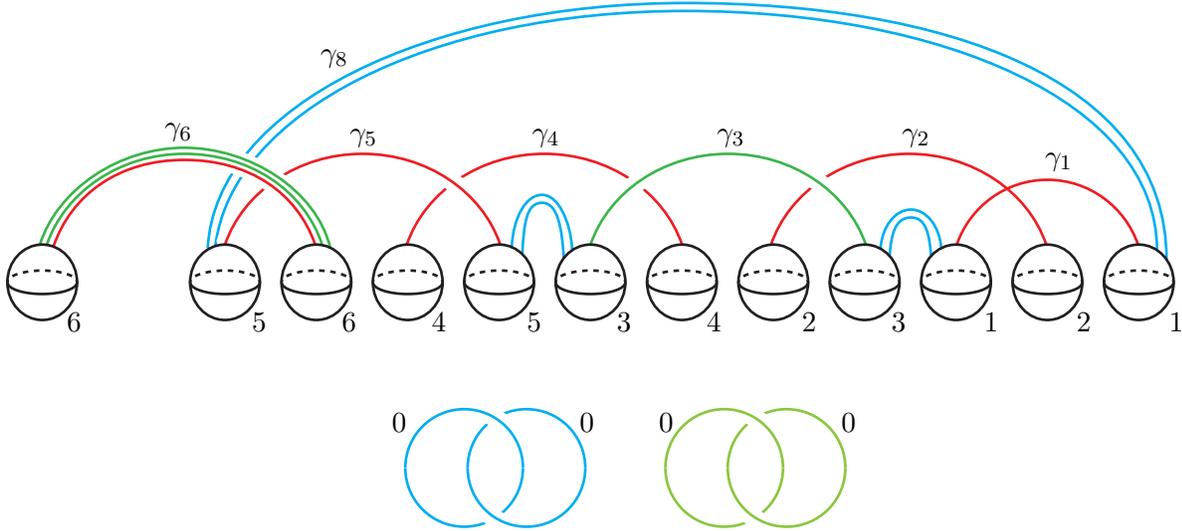
\begin{figure}[htb]{
\begin{overpic}
{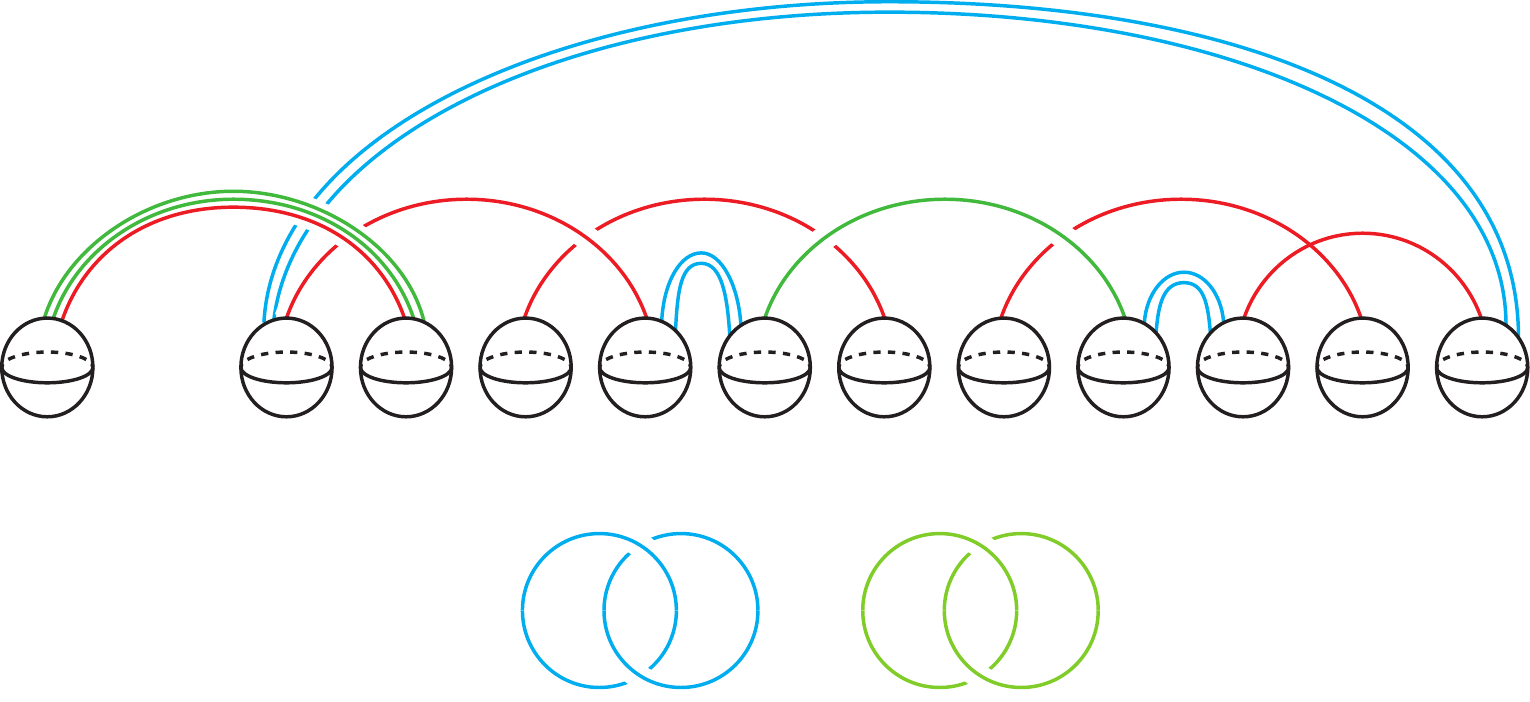}
\put(440,78){$1$}
\put(405, 78){$2$}
\put(370,78){$1$}
\put(335,78){$3$}
\put(301,78){$2$}
\put(265,78){$4$}
\put(231,78){$3$}
\put(197,78){$5$}
\put(161,78){$4$}
\put(127,78){$6$}
\put(93,78){$5$}
\put(23,78){$6$}
\put(119,180){$\gamma_{8}$}
\put(60,152){$\gamma_6$}
\put(130,150){$\gamma_5$}
\put(199,150){$\gamma_4$}
\put(269,150){$\gamma_3$}
\put(339,150){$\gamma_2$}
\put(393,141){$\gamma_1$}
\put(146, 40){$0$}
\put(217, 40){$0$}
\put(247, 40){$0$}
\put(316, 40){$0$}
\end{overpic}}
\caption{The top figure is a handle presentation describing manifolds in Proposition~\ref{thisone} when the genus is $3$. All $2$--handles have framing $\pm1$ depending on if they are chiral or achiral. The 1--handles are paired as indicated by the numbers. The bottom figure is the result of canceling all the 1--handles with the vanishing cycles $\gamma_1$ to $\gamma_5$ and the chiral $\gamma_6$. }
\label{fig:stein2}
\end{figure}
One may easily check that this 4--manifold is $X=S^2\times S^2\# S^2\times S^2$ with a disk removed. Saeki \cite{Saeki87} has shown that $S^5$ has an open book decomposition with $X$ as its page.  The result follows from   Lemma~\ref{maintopembed}.
\end{proof}

\begin{question}
Is there an open book for $S^5$ from which one can give a spun embedding of all $3$--manifolds?
\end{question}

\begin{remark}
In unpublished work Atsuhide Mori has sketched an idea to construct such spun embeddings into $S^5$ using the open book on $S^5$ with $D^4$ pages. Moreover, while completing a draft of this paper the authors were informed by Dishant Pancholi, Suhas Pandit, and Kuldeep Saha that they could construct similar such embeddings.  It would still be interesting to know if one could construct a Lefschetz fibration structure on a page of an open book for $S^5$ to use the techniques in this paper to find embeddings of all 3--manifolds into $S^5$, {\em cf.\ }Remark~\ref{difficulty}. 
\end{remark}

 \section{Contact embeddings}
 In the first subsection, we discuss an obstruction to embedding contact manifolds that generalizes Kasuya's theorem discussed in the introduction. In the following subsection, we prove Theorem~\ref{mainembedthm} that there is a Stein fillable contact structure on the twisted $S^3$-bundle over $S^2$ into which all closed, oriented contact 3--manifolds embed. In the final subsection, we make an observation about embedding contact manifolds into the standard contact structure on $S^5$. 

\subsection{Obstructions to contact embeddings}\label{sec:obstruction}
We begin with a simple lemma about co-dimension $2$ contact embeddings. 
\begin{lemma}\label{basicobstruction}
Let $e:(M,\xi)\to (W,\xi')$ be a co-dimension $2$ contact embedding with trivial normal bundle. Then $c_1(\xi)=e^*c_1(\xi')$.
\end{lemma}
\begin{proof}
The standard contact neighborhood theorem says that $e(M)$ has a neighborhood contactomorphic to a neighborhood of $M\times \{0\}$ in $M\times \C$ with the contact form $\alpha+r^2\, d\theta$, where $\alpha$ is a contact form for $\xi$ and $(r,\theta)$ are polar coordinates on $\C$. So we see that $\xi'$ along $M$ has a splitting as a complex bundle into $\xi\oplus \C$ and thus $c_1(e^*\xi')=c_1(\xi)+c_1(\C)=c_1(\xi)$.
\end{proof}
We note that Kasuya's result, Theorem~\ref{cantembed}, follows from this.
\begin{proof}[Proof of Theorem~\ref{cantembed}]
We first note that the result in \cite{Kasuya14} actually says that if $(M^{2n-1},\xi)$ embeds in $(W^{2n+1},\xi')$, and $H^2(W;\Z)=0$, then $c_1(\xi)=0$, but the proof gives the stronger version stated in the introduction. To prove the result form the above lemma notice that if $e(M)$ is trivial in homology then it bounds a hypersurface $\Sigma$ that in turn trivializes the normal bundle of $e(M)$, \cite{Kirby89}.  The result now clearly follows. 
\end{proof}
We notice Lemma~\ref{basicobstruction} has other consequences too. Specifically, when searching for a contact manifold into which all contact $3$--manifolds embed one might consider the simplest ones first, that is the product of two spheres. 
\begin{corollary}\label{onlys2s3}
If all contact $3$--manifolds $(M,\xi)$ embeds with trivial normal bundle in a manifold $W^5$ with a co-oriented contact structure $\xi'$ where $W$ is a product of two or fewer spheres, then $W=S^3\times S^2$ and the $c_1(\xi')=\pm 2h$ where $h\in H^2(W,\Z)=\Z$ is a generator.  
\end{corollary}
\begin{proof}
The only products of spheres are $S^5$, $S^1\times S^4$ and $S^2\times S^3$. In the first two examples notice that an embedding of a $3$--manifold always realizes the trivial homology class and thus Lemma~\ref{basicobstruction} implies the Chern class of $\xi$ is trivial. Now consider a co-oriented contact structure $\xi'$ on $W=S^2\times S^3$ choosing a generator $h$ of the second cohomology we see that $c_1(\xi')=2k h$ for some integer $k$ since the mod $2$ reduction of $c_1(\xi')$ is $w_2(\xi')=w_2(TW)=0$. Now given a contact embedding $e:(M,\xi)\to (W,\xi')$ with trivial normal bundle 
we see $c_1(\xi)=2k e^*(h)$. In particular if $k=0$ the Chern class is trivial, if $k\not=0$ then $c_1(\xi)$ is divisible by $2k$. Since we have contact structures on $3$--manifolds with $c_1$ divisible by only $2$ we see that $k=\pm1$. (Notice that the Chern class of any co-oriented contact structure on a $3$--manifold is even for the same reason as for such structures on $W$.)
\end{proof}

\subsection{Contact embeddings in the twisted $S^3$-bundle over $S^2$}

We are now ready to prove our main embedding theorem, Theorem~\ref{mainembedthm}, that says there is a Stein fillable contact structure $\xi$ on the twisted $S^3$-bundle over $S^2$ into which all contact $3$--manifolds embed. 
\begin{proof}[Proof of Theorem~\ref{mainembedthm}]
Let $X$ be the disk bundle over $S^2$ with Euler number $-3$. From Example~\ref{idmonodromy} we see that the manifold $M_{(X,id_X)}$ is the twisted $S^3$-bundle over $S^2$, $S^2\widetilde\times S^3$. We denote the binding of this open book by $Y$ and the fibration $(S^2\widetilde\times S^3)-Y \to S^1$ by $\pi'$. As discussed in Proposition~\ref{exampleprop} there is a Lefschetz fibration $f:X\to D^2$ with fiber genus greater than $2$ and by Theorem~\ref{leftosym} there is an exact symplectic structure (in fact Stein structure) with the properties listed in the theorem (in particular the fibers of $f$ are symplectic). Thus, by Example~\ref{examplestein} the contact structure $\xi_{(X,id_X)}$ on $S^2\widetilde\times S^3$ is Stein fillable by $X\times D^2$. Recall as $\xi_{(X,id_X)}$  is supported by the open book $(Y,\pi')$ we know that there is a contact 1--form $\alpha$ for $\xi_{(X,id_X)}$ so that $d\alpha$ is a symplectic form on each page of the open book and $\alpha$ restricted to the binding is a (positive) contact form on the binding. 

Now given a contact $3$-manifold $(M,\xi)$ we know that $\xi$ can be supported by an open book decomposition $(\Sigma,\phi)$ where $\Sigma$ has connected boundary. We denote the binding of this open book by $B$ and the fibration of its complement by $\pi$. Proposition~\ref{exampleprop} gives us a Lefschetz fibration of $X$ with generic fiber $\Sigma$ and vanishing cycles $\gamma_1,\ldots, \gamma_{2g+1}$ (from Figure~\ref{fig:surface}). As it is well-known that Dehn twists about these curves generate the mapping class group of a surface, there is a smooth embedding of $e:M\to S^2\widetilde\times S^3$ by Lemma~\ref{maintopembed}. We will show by exerting more care in the proof of Lemma~\ref{maintopembed} one may show that $e$ sends each page of the open book for $M$ to a symplectic surface in a page of the open book $(Y,\pi')$ (where the symplectic structure comes from $d\alpha$) and the binding of $M$ maps to a transverse knot in the binding of $Y$. If we can do this then clearly $e^*\alpha$ will define a contact structure supported by $(B,\pi)$ and since $\xi$ is also supported by $(B,\pi)$ we see that $e$ is (isotopic to) a contact embedding. 

Recall in Step 1 of the proof of Lemma~\ref{maintopembed} we constructed an embedding 
\[
\widetilde{e}:\Sigma\times [0,1]\to X\times [0,1]: (p,t)\mapsto (i(p),t)
\]
that would descent to an embedding of the mapping tori $T_\phi$ to $T_{id_X}$. Now the contact form on $T_{id_X}$ is given by $k\, dt + \beta$ where $\beta$ is the exact symplectic form on $X$. We also know that each fiber of $f:X\to D^2$ is symplectic and that $\widetilde{e}$ send $\Sigma$ to fibers of $f$. Thus, we see that $k\, dt + \beta$ pulls back to a contact form on $T_\phi$ by Remark~\ref{usefulcomments}.

Now in Step 2 of the proof of Lemma~\ref{maintopembed} we extended the embedding $e$ over a collar neighborhood of the boundary to make $\widetilde{e}$ (and hence $e$) independent of $t$. Recall in the proof we write a neighborhood of one of the components of $f^{-1}((\text{int} D^2))$) as $[a,b]\times S^1\times \text{int} (D^2)$ and in these coordinates we can use  Theorem~\ref{leftosym} to say that the exact symplectic structure on $X$ is given by $e^s(d\varphi + r^2\, d\theta$) (here we are use $\varphi$ as the angular coordinate on $S^1$, $s$ as the coordinate on $[a,b]$, and $(r,\theta)$ as polar coordinates on $D^2$). Now in Step 2 of the proof of Lemma~\ref{maintopembed}  we considered the extension of $\widetilde{e}$ by 
\[
([b,c]\times S^1)\times [0,1]\to ([b,c]\times S^1\times\text{int} (D^2))\times [0,1]: (s,\varphi, t)\mapsto (s,\varphi, \gamma_t(s),t),
\]
and thus pulling  $k\, dt + e^s(d\varphi + r^2\, d\theta$) back by this map yields 
\[
k\, dt + e^s\left(d\varphi + f_t^2(s)\, \frac{\partial g_t(s)}{\partial s} \, ds\right),
\] 
where $\gamma_t(s)$ expressed in polar coordinates is given by $(f_t(s), g_t(s))$. This is clearly a contact form if either $k$ is chosen sufficiently large or if $c-b$ is sufficiently large that the derivatives of $f_t$ and $g_t$ are sufficiently small. Moreover, if we choose the $g_t(s)$ to be constant near $c$ then it is simply $k\, dt + e^s\, d\varphi$ near each boundary component of $\Sigma$. 

Since we have a standard model for our embedding near $\partial T_\phi$, it is now a simple matter to see that the extension of $e$ over the neighborhoods of the binding has the desired properties. 
\end{proof}

\subsection{Contact embeddings in $(\R^5, \xi_{std})$}
We notice that the same proof as the one given for Theorem~\ref{mainembedthm} also yields the following results. 
\begin{theorem}
If $(M,\xi)$ is a contact $3$--manifold supported by an open book with hyperelliptic monodromy then $(M,\xi)$ embeds in $(S^5,\xi_{std})$. 
\end{theorem}
We note that this theorem is known, see for example \cite{EtnyreFurukawa17}, and can be used to show among other things that all tight contact structures with $c_1=0$ on $L(p,q)$ with $p$ odd or $p$ even and $q=1$ or $p-1$ can be embedded in $(S^5,\xi_{std})$. 

We can similarly give a criterion guaranteeing a contact structure embeds in a Stein fillable contact structure on $S^2\times S^3$.
\begin{theorem}
If $(M,\xi)$ is a contact $3$--manifold supported by an open book $(\Sigma, \phi)$ with $\Sigma$ a genus $g$ surface with one boundary component and $\phi$ a composition of Dehn twists about the curves $\gamma_1,\ldots, \gamma_{2g}$ and $\gamma_{2g+2}$, then $(M,\xi)$ embeds in a Stein fillable contact structure on $S^2\times S^3$. 
\end{theorem} 

\begin{remark}\label{difficulty} We note that Dehn twists around $\gamma_1,\ldots, \gamma_{2g}$ and $\gamma_{2g+2}$ do not generate the mapping class group of $(\Sigma, \partial \Sigma)$, where $\Sigma$ is a genus $g$ surface with one boundary component. One way to see this is as follows: Consider the spin structure on $\Sigma$ given by the quadratic form $q: H_1(\Sigma; \mathbb{Z}_2) \to \Z_2$ enhancing the intersection product, that assigns $q(\gamma_i)=0$ for $i=1,\ldots, 2g$. It automatically follows that $q(\gamma_ {2g+2})=0$ since $\gamma_1,\gamma_3,\gamma_5$ and $\gamma_{2g+2}$ bound a genus 0 surface (see \cite{johnson} for the relation between spin structures and quadratic forms). Now, Dehn twists around $\gamma_{i}$ preserve this spin structure. On the other hand, it is known that the action of the mapping class group on the set of spin structures on a surface has exactly 2 orbits distinguished by the Arf invariant of $q$, \cite{atiyah, johnson}. 

	A similar argument also implies that one cannot hope to find a set of curves $\{v_i\}$ such that Dehn twists around them generate the mapping class group and also such that the curves $\{v_i\}$ are the vanishing cycles for a Lefschetz fibration on some Stein surface with fiber $\Sigma$, whose total space is a page of an open book on $S^2 \times S^3$. Indeed, since $S^2 \times S^3$ is spin, such a construction would induce a spin structure on $\Sigma$ with $q(v_i)=1$ (as the vanishing cycles bound thimbles, see \cite{stipsicz}), thus the image of the Dehn twists around $v_i$ in $Sp_{2g}(\mathbb{Z})$ can only generate the theta group, a certain subgroup of the $Sp_{2g}(\mathbb{Z})$ generated by anisotropic transvections, see \cite{johnsonmillson}.
\end{remark}

 \section{Contact embeddings in overtwisted contact structures}
In this section we prove Theorem~\ref{otembed} which shows that any contact $3$--manifold contact embeds in a unique overtwisted contact structure on $S^2\times S^3$. To this end we first show how to ``explicitly", in terms of handle decompositions, smoothly embed a $3$--manifold in $S^2\times S^3$ and $S^5$. 

\subsection{Topologically embedding $3$--manifolds in $S^5$ and $S^2\times S^3$}
We begin by embedding a 3--manifold in $S^5$.
Given an oriented $3$--manifold $M$ we can find a handlebody decomposition of the form
\[
M=h^0\cup (h^1_1\cup \ldots h^1_g) \cup (h^2_1\cup \ldots h^2_g)\cup h^3.
\]
where $h^i_j$ is an $i$--handle and the handles are attached in the order in which they appear above. We will build a handlebody structure on $S^5$ in which we see the above handle decomposition for $M$. 

\noindent{\em Step I: Handle decomposition of $M\times D^2$.}
Consider $M\times D^2$. This is a 5-manifold with an analogous handle decomposition
\[
M\times D^2=H^0\cup (H^1_1\cup \ldots H^1_g) \cup (H^2_1\cup \ldots H^2_g)\cup H^3,
\]
where each $H^i_j$ is simply $h^i_j\times D^2$. 

\noindent
{\em Step II: Cancel the 1--handles.} We would now like to cancel the 1--handles $H^1_j$. To this end we notice that $\partial (M\times D^2)=M\times S^1$ and the co-cores of the 1--handles intersect this boundary in solid tori $D^2\times S^1$ where $D^2$ is the stable manifold of $h^1_j$ in $M$. We can now take a simple closed curve $\gamma_i$ in $M$ that intersects the $D^2$ one time then attach a $2$--handle to $M\times D^2$ along $\gamma_i\times \{\theta\}$. This clearly cancels the 1--handle $H^1_j$ and we can choose the $\gamma_i$ disjoint in $M$ so all the extra $2$--handles are attached disjointly. Call the new manifold $W$ and notice that $W$, after canceling handles has a handle decomposition
\[
W=H^0\cup (H^2_1\cup \ldots H^2_g)\cup H^3.
\]

\noindent
{\em Step III: Cancel the 2--handles.}
Now for each $H^2_j$ we are going to choose a sphere $S_j$ in $\partial W$ intersecting the co-core of $H^2_j$ once. To find these spheres notice that $\partial (H^0\cup (H^2_1\cup \ldots H^2_g))$ is the connected sum of $g$ copies of $S^2\times S^2$ (this is simply because all framings on the $2$--handles are 0 since the 3--dimensional 2--handles had trivial framing and the attaching spheres are $S^1$'s which are all isotopic in $S^4$). The belt spheres of the 2-handles are all of the form  $\{p\}\times S^2$ in one of the $S^2\times S^2$ summands. Thus there are obvious complementary spheres $S_j$ (of the form $S^2\times \{p\}$). These spheres can be assumed to be disjoint from the attaching region for our 3--handle $H^3$. So we can attach 3--handles to $W$ along the $S_j$ to get a 5--manifold $W'$ with handle decomposition 
\[
W'=H^0\cup H^3.
\]

\noindent
{\em Step IV: Cancel the 3--handle and complete the embedding.} Add a 2-handle $H^2$ to the belt sphere of $H^3$ with framing 0 (and isotope the attaching sphere so it is in $\partial H^0$). Notice that $H^0\cup H^2$ has boundary $S^2\times S^2$ where the belt sphere for the 2--handle is $S_b=S^2\times \{p\}$ and $S_c=\{p\}\times S^2$ is the union of a copy of the core of the 2-handle and a meridional disk for the attaching sphere $S$ for the 3--handle $H^3$. Notice that $S$ intersects $S_c$ exactly once and is disjoint from $S_b$. Thus it is clear $S$ is in the homology class of $S_b$. By Gabai's 4--dimensional light bulb theorem, \cite{Gabai17pre}, $S$ is isotopic to $S^2\times \{p\}$. Hence, we can take the attaching sphere for $H^3$ to be an unknotted 2--sphere in $S^4=\partial H^0$ and the attaching sphere for $H^2$ to be a meridional $S^1$ to this sphere. We can now clearly attach a cancelling 3-handle for $H^2$ and a cancelling 4--handle for $H^3$ resulting in a new manifold $W''$ with handle decomposition 
\[
W''=H^0
\]
to which we can clearly attach a 5--handle to get $S^5$ with $M$ as a submanifold. 

We now turn to embedding $M$ into $S^2\times S^3$. The first 2 steps are the same. 

\noindent
{\em Modified Step III: Cancel the 2--handles.} As before, we would now like to cancel all the $2$--handles, but since $S^2\times S^3$ does have a $2$--handle in its description we will first add this one and then cancel all the ones coming from the $3$--manifold. So attach $H^2_{g+1}$ to a circle in $M\times \{\theta'\}$ with framing 0. Later we will need some flexibility in this step so for each $j=1,\ldots g$ we choose an integer $k_j$. Now for each $H^2_j$ we are going to choose a sphere $S_j$ in $\partial W$ that intersects the co-core of $H^2_{g+1}$ transversely in $|k_j|$ points (and positively if $k_j>0$ and negatively if $k_j<0$) and intersecting the co-core of $H^2_j$ once and positively. As before, to find these spheres notice that $\partial (H^0\cup (H^2_1\cup \ldots H^2_g)\cup H^2_{g+1})$ is the connected sum of $g+1$ copies of $S^2\times S^2$. Moreover if we consider the $S^2\times S^2$ summand coming from $H^2_{g+1}$ we see that the boundary of the co-core of the handle is simply $\{p\}\times S^2$. Thus given $j$ we can take $k_j$ spheres $S^2\times \{q_1,\ldots, q_k\}$ and then tube them together. This gives a sphere $S_j'$ that has trivial normal bundle and intersects the co-core of $H^2_{g+1}$ the correct number of times. We can take a simpler sphere $S_j''$ that intersects the co-core of $H^2_j$ exactly once (for future use notice that this sphere might run over some of the previously cancelled 1 and $2$--handles). Now choose an arc in $\partial W$ connecting $S_j'$ to $S_j''$. Tubing these spheres together gives the desired attaching spheres $S_j$ for our $3$--handles. After attaching these $3$--handles we have a 5--manifold $W''$ with handle decomposition
\[
W'=H^0\cup H^2_{g+1} \cup H^3.
\]

\noindent
{\em Modified Step IV: Cancel the 3--handle and complete the embedding.} 
We can now cancel the 3--handle as in Step IV above since everything there could be assumed to take place in the complement of a ball where $H^2_{g+1}$ is attached. After cancelling we are left with a 5--manifold $W''$ with handlebody decomposition 
\[
W''=H^0\cup H^2_{2g+1}
\]
and $W''\cong S^2\times D^3$. Finally we attach a $3$--handle $H^3$ and a $5$--handle $H^5$ to get $S^2\times S^3$ and we clearly have $M$ as a submanifold. 

\subsection{Embedding contact $3$--manifolds in the overtwisted contact structure on $S^2\times S^3$}
\begin{proof}[Proof of Theorem~\ref{otembed}]
Given a contact $3$--manifold $(M,\xi)$, let $\alpha$ be a contact $1$--form for $\xi$. Notice that $M\times D^2$ has a contact structure $\xi'$ given by $\alpha+r^2\, d\theta$. Using the smooth embedding of $M$ into $S^2\times S^3$ from the previous section we will show that $\xi'$ can be extended to an almost contact structure $\eta$ on $S^2\times S^3$ with Chern class twice a generator of $H^2(S^2\times S^3;\Z)$. Then Theorem~\ref{otex} says $\eta$ may be homotoped to an actual contact structure $\xi''$ on $S^2\times S^3$ that agrees with $\xi'$ on $M\times D^2$. Thus we have a contact embedding of $(M,\xi)$ into $(S^2\times S^3,\xi'')$. Moreover since $H^2(S^2\times S^3; \Z)$ has no $2$--torsion the proof of the theorem is then complete by noting that Theorem~\ref{otex} says $\xi''$ is uniquely determined by the fact that its Chern class is twice a generator in co-homology as explained in the beginning of Section~\ref{sec:ot}.

We are left to construct the almost contact structure $\eta$. Recall from Section~\ref{sec:ot} that almost contact structures are sections of the $SO(5)/U(2)$-bundle associated to the tangent bundle of the manifold. It is well-known that the tangent bundle of $S^2\times S^3$ is trivial. So sections of this bundle are equivalent to functions from $S^2\times S^3$ to $SO(5)/U(2)$. Recall that $\pi_k(SO(5)/U(2))=0$ for $k\leq 5$ except for $k=2$ when it is $\Z$ (this is because $SO(5)/U(2)\cong \C P^3$). So by obstruction theory we can extend $\xi'$ on $M\times D^2$ to $\eta$ on the 2--skeleton. 
Extend $\eta$ over $H^2_{g+1}$ so that it maps to a generator of $\pi_1(SO(5)/U(2))$ (recall that $H^2_{g+1}$ can be attached to a circle in a small ball in $\partial W$). The attaching region for each additional 3--handle is an $S^2$ and $\eta$ will send it to some element of $\pi_1(SO(5)/U(2))$. By choosing the $k_j$ in the Modified Step III above correctly we assume $\eta$ sends this $S^2$ to zero. Thus we can extend $\eta$ over the 3--handles added in Step III. The 2 and 3--handle added in the original Step IV cancel each other so we can extend $\eta$ over them too.  The 3--handle $H^3_f$ attached in the Modified Step IV is attached to the boundary of the co-core of the added 2--handle, so $\eta$ restricted to the attaching sphere of $H^3_f$ is null-homotopic (since it extends over the co-core of the 2--handle) and so can be assumed to be constant on the attaching region. Hence, we can extend $\eta$ over $H^3_f$. As the remaining handles have index above 3 we can extend $\eta$ over them and thus to all of $S^2\times S^3$

If we let $\eta'$ be the almost contact structure that is the constant map to $SO(5)/U(2)$ then the obstruction theory difference class $d(\eta,\eta')$ satisfies 
\[
2d(\eta,\eta')=c_1(\eta)-c_1(\eta'),
\]
see \cite{Geiges08, Hamilton08}.
Since $c_1(\eta')=0$ and $H_2(S^2\times S^3)$ is generated by $H^2_{g+1}$ which is sent to the generator of $\pi_1(SO(5)/U(2))$ we see $c_1(\eta)=\pm 2\in H^2(S^2\times S^3)$. 

We note that as the embedding constructed above is disjoint from the final 3--handle in $S^2\times S^3$ it is trivial in $H_3(S^2\times S^3)$ and thus has trivial normal bundle \cite[Theorem~VIII.2]{Kirby89}, thus Corollary~\ref{onlys2s3} can be applied. 
Since there are contact structures on 3--manifolds with first Chern classes twice a generator, Corollary~\ref{onlys2s3} says that any contact structure on $S^2\times S^3$ into which all contact 3--manifolds embed must have first Chern class $\pm 2$. 
Since overtwisted contact structures on $S^2\times S^3$ are determined by their first Chern class there are exactly two overtwisted contact structures, up to isotopy, satisfying this requirement. Moreover, since $S^2\times S^3$ admits an orientation reversing diffeomorphisms that acts as minus the identity on the second homology, these two contact structures are contactomorphic. Thus we see that there is exactly one overtwisted contact structure on $S^2\times S^3$ (up to contactomorphism) into which all contact 3--manifolds embed. 
\end{proof}
\def\cprime{$'$} \def\cprime{$'$}


\end{document}